\documentclass[11pt,hyp,]{nyjm}
\usepackage{hyperref}
\hypersetup{nesting=true,debug=true,naturalnames=true}
\usepackage{graphicx,amssymb,upref, amsmath}

\let\<\langle
\let\>\rangle

\let\uml\"

\title{Asymptotically Orthonormal Basis and Toeplitz operators}

\author[Fricain]{Emmanuel Fricain}
\address{Laboratoire Paul Painlev\'e, Universit\'e Lille 1, 59 655 Villeneuve d'Ascq C\'edex }
\email{emmanuel.fricain@math.univ-lille1.fr}

\author[Rupam]{Rishika Rupam}
\address{Laboratoire Paul Painlev\'e, Universit\'e Lille 1, 59 655 Villeneuve d'Ascq C\'edex}
\email{rishika.rupam@math.univ-lille1.fr}

\thanks{The authors were supported by Labex CEMPI (ANR-11-LABX-0007-01)}


\date{}

\newcommand\HH{{\mathcal H}}

\newcommand\R{{\mathbb R}}

\newtheorem{thm}{Theorem}[section]
\newtheorem{lem}[thm]{Lemma}

\begin{document}
\maketitle

\begin{abstract}
Recently, M. Mitkovski \cite{mitkovski2017basis} gave a criterion for the basicity of a sequence of complex exponentials in terms of the invertibility properties of a certain naturally associated Toeplitz operator, in the spirit of the celebrated work of Khrushch{\"e}v--Nikolski--Pavlov \cite{khruscev1981unconditional}. In our paper, we extend the results of Mitkovski to model spaces associated with meromorphic inner functions and we also give an analogue for the property of being an asymptotically orthonormal basis.
\end{abstract}

\section{Introduction}
The classical theory of Fourier series says that the trigonometric system $(e^{int})_{n\in\mathbb Z}$ forms an orthonormal basis of $L^2(0,2\pi)$. It is natural to ask what happens if we replace the trigonometric system by another systems of complex exponential $(e^{i\lambda_n t})_{n\in\mathbb Z}$, $\Lambda=(\lambda_n)_{n\in\mathbb Z}\subset\mathbb R$. Indeed this problem has a very long and rich history and finds its roots in the work of Paley--Wiener and Levinson. There at least two directions of approach to this problem. First, we can try to find small perturbations of the trigonometric systems which retains the desired basis properties of expansion. This problem of perturbation in fact gave rise to a whole direction of research, culminating with the beautiful theorem of Ingham--Kadets. See \cite{Young} for more details. The second direction is to find characterization of discrete sequences $\Lambda=(\lambda_n)_{n\in\mathbb Z}\subset\mathbb R$ such that $(e^{i\lambda_n t})_{n\in\mathbb Z}$ forms a Riesz basis/asymptotically orthonormal basis of $L^2(0,2\pi)$. In the seventies, an approach using Toeplitz operators was developed by Douglas, Sarason and Clark. Not only did it allow the recapture of all the classical results but also permitted Khrushch{\"e}v--Nikolski--Pavlov \cite{khruscev1981unconditional} to give the solution of the Riesz basis problem for complex exponentials. We briefly explain the basic idea of the Toeplitz approach. Since the map $f(t)\longmapsto f(t)e^{-t}$ is an isomorphism on $L^2(0,2\pi)$, it is clear that $(e^{i\lambda_n t})_{n\in\mathbb Z}$ forms a Riesz basis of $L^2(0,2\pi)$ if and only if $(e^{i\lambda'_n t})_{n\in\mathbb Z}$ does, where $\lambda'_n=\lambda_n+i$, $n\in\mathbb Z$. If we apply the inverse of the Fourier transform, it is easy to see that $(e^{i\lambda'_n t})_{n\in\mathbb Z}$ is a Riesz basis for $L^2(0,2\pi)$ if and only if $(k_{\lambda'_n}^{\Theta_{2\pi}})_{n\in\mathbb Z}$ is a Riesz basis for $\mathcal K_{\Theta_{2\pi}}$. Here $\mathcal K_{\Theta_{2\pi}}$ is the model space associated to $\Theta_{2\pi}(z)=e^{2i\pi z}$ and $k_{\lambda'_n}^{\Theta_{2\pi}}$ is its reproducing kernel. Then, the crucial idea is to observe that the basis properties of the reproducing kernels in $K_{\Theta_{2\pi}}$ are encoded in the invertibility property of the Toeplitz operator $T_{\overline{B_{\Lambda'}}\Theta_{2\pi}}$, where $B_{\Lambda'}$ is the Blaschke product associated to $\Lambda'=(\lambda_n')_{n\in\mathbb Z}$. If we want to apply the same method to study the asymptotically orthonormal basis property, we are faced with a difficulty because the latter property is not preserved by isomorphism, and consequently we cannot shift $\Lambda$ to $\Lambda'$. But recently, Mitkovski \cite{mitkovski2017basis} has shown that we can adapt the approach of Khrushch{\"e}v--Nikolski--Pavlov avoiding the use of the translated sequence but instead exploiting the ideas of Makarov--Poltoratski's Toeplitz approach to the completeness problem of exponential systems. In this paper, we will show that one can generalize Mitkovski's results to more general model spaces and we also give an analogous result for the asymptotically orthonormal basis property. \\

In Section 2, we start with some general preliminaries on Riesz basis, asymptotically orthonormal basis and meromorphic inner functions. As we will see, our characterization of asymptotically orthonormal basis involves the condition that a certain Toeplitz operator is unitary modulo compact (meaning that it can be written as the sum of a unitary operator and a compact one). In Section 3, we thus revisit some classical results on  Toeplitz operators. In particular, we explain how to derive a result of Douglas (characterizing the Toeplitz operators which are unitary modulo compact) from Hartman's theorem on compactness of Hankel operators and Devinatz--Widom's theorem on invertibility of Toeplitz operators. These results in the context of Hardy space of the unit disc can be found for instance in Peller \cite{Peller} and Nikolski \cite{nikolski2009functions}. Here we need the analogue for the upper-half plane and we provide the details with an aim to make our presentation complete. In Section 4, we give a characterization for Riesz basis property and in the last section, we discuss the asymptotically orthonormal basis property for systems of reproducing kernels in the model spaces.

\section{Preliminaries}

\subsection{Definition of Riesz sequences and AOB}
Let $\HH$ be a Hilbert space, $\mathfrak X=(x_n)_{n\geq 1}$ be a sequence of vectors in $\HH$. We recall that $\mathfrak X$ is said to be:
\begin{enumerate}
\item[(a)] \emph{minimal} if for every $n\geq 1$, 
$$
x_n\not\in\mbox{span}(x_\ell:\ell\not=n),
$$
where $\mbox{span}(\dots)$ denotes the closure of the finite linear combination of $(\dots)$;
\item[(b)] A \emph{Riesz sequence} (abbreviated RS) if there exists two positive constants $c,C$ such that 
\begin{equation}\label{eq:defn-RS}
c\sum_{n\geq 1}|a_n|^2\leq \|\sum_{n\geq 1}a_n x_n\|^2_{\HH}\leq C\sum_{n\geq 1}|a_n|^2,
\end{equation}
for every finitely supported sequence of complex numbers $(a_n)_n$;
\item[(c)] An \emph{asymptotically orthonormal sequence} (abbreviated AOS) if there exists $N_0\in\mathbb{N}$ such that for all $N\geq N_0$ there are positive constants $c_N,C_N$ verifying
\begin{equation}\label{eq:AOB}
c_N\sum_{n\geq N}|a_n|^2\leq \|\sum_{n\geq N}a_n x_n\|^2_{\HH}\leq C_N\sum_{n\geq N}|a_n|^2,
\end{equation}
for every finitely supported sequence of complex numbers $(a_n)_n$ and $\lim_{N\to \infty}c_N=1=\lim_{N\to \infty}C_N$;
\item[(d)] An \emph{asymptotically orthonormal basic sequence} (abbreviated AOB) if it is an AOS with $N_0=1$; 
\item[(e)] A \emph{Riesz basis} for $\mathcal H$ (abbreviated RB) if it is a complete Riesz sequence, that is a Riesz sequence satisfying 
$$
\mbox{span}(x_n:n\geq 1)=\HH.
$$
 \end{enumerate}
It is easy to see that $(x_n)_{n\geq 1}$ is an AOB if and only if it is an AOS as well as a RS. Also, $(x_n)_{n\geq 1}$ is an AOB if and only if it is minimal and an AOS. The reader should pay attention to the fact that AOB does not imply completeness; an AOB is a basis for its span but not necessarily for the whole space.

\subsection{Meromorphic inner functions and model spaces.}
A function $\Theta:\mathbb C_+\longrightarrow \mathbb C$ is said to be {\em inner} on the upper-half plane $\mathbb C_+=\{z\in\mathbb C:\Im(z)>0\}$ if $\Theta$ is analytic and bounded on $\mathbb C_+$ and if its radial limits are of modulus one almost everywhere on $\mathbb R$. We say that the inner function $\Theta$ is a meromorphic inner function (abbreviated by MIF) if $\Theta$ admits a meromorphic extension to the whole complex plane. A well-known theorem by Riesz and V. Smirnov says that all MIF functions have the form 
$$
\Theta(z)=B(z)e^{iaz}
$$
where $a\geq 0$ and $B$ is the Blaschke product formed with the zeros of the function $\Theta$ given by $Z=(z_n)_{n\geq 1}$, where $|z_n|\to\infty$ and satisfy the convergence criterion
$$
\sum_{n\geq 1}\frac{\Im z_n}{1+|z_n|^2}<\infty.
$$
In particular, using this representation, it is not difficult to see that $\Theta$ will admit an analytic extension through the real line.

To each inner function $\Theta$, we associate a {\em model space}
$$
\mathcal K_\Theta=\mathcal H^2\ominus\Theta\mathcal H^2=\mathcal H^2\cap \Theta\overline{\mathcal H^2},
$$
where $\mathcal H^2$ is the Hardy space on $\mathbb C_+$. It is well known that if $\Theta$ is a MIF, then every function $f\in\mathcal K_\Theta$ can be extended analytically through the real line. In the case when $\Theta$ is a MIF with $\Theta'\in L^\infty(\mathbb R)$, we will need some uniform $L^2$--estimate on this extension. We start with the following result appearing in \cite[Lemma 2]{dyakonov2002differentiation}. For completeness, we provide a proof. 

\begin{lem}\label{Lem:mean-value-thm}
Let $\Theta$ be a MIF inner function such that $\Theta' \in L^\infty(\R)$. Then,  for every $\varepsilon>0$, we have
$$\displaystyle\inf_{0\leq \Im z <\varepsilon} |\Theta(z)|\geq 1-\varepsilon \|\Theta'\|_\infty.
$$
\end{lem}  
\begin{proof}
Since $\Theta$ is analytic on the real line, we can apply the mean value theorem. Then, for all $x,h \in \R$,
$$|\Theta(x+h)-\Theta(x)|\leq M|h|,$$
where $M=\|\Theta'\|_\infty$.
Consider the function $z \rightarrow \Theta(z+h)-\Theta(z)$. This function is the Poisson integral of the boundary function $x\rightarrow \Theta(x+h)-\Theta(x),$ which permits to extend the latter estimate in $\mathbb C_+$. In other words, for all $h\in\mathbb R_+$ and $z\in\mathbb C_+$, we have
$$
|\Theta(z+h)-\Theta(z)|\leq M|h|.
$$ 
But then by the definition of the derivative,
$$|\Theta'(z)|\leq M, $$
for all $z \in \mathbb C_+$. Note that we have proved that $\Theta' \in H^\infty(\mathbb C_+)$. Now letting $z=x+iy$, $0<y<\varepsilon$, we have
$$
|\Theta(z)-\Theta(x)|\leq |z-x|\displaystyle\sup_{0\leq \Im z <\varepsilon} |\Theta'(z)|\leq \varepsilon M,
$$
which gives
\begin{equation}\label{eq2-Theta-minoration1}
|\Theta(z)|\geq 1-\varepsilon M.
\end{equation}
\end{proof}
Now, let $f\in \mathcal K_\Theta$ and $0<\varepsilon<\|\Theta'\|_\infty^{-1}$. We can write $f(z)=\Theta(z)\overline{h(z)}$, $z\in\mathbb C_+$, for some $h\in\mathcal H^2$. 
Then, if we define for $-\varepsilon<\Im(z)<0$, 
\begin{equation}\label{eq1-analytic-continuation}
f(z)=\frac{1}{\overline{\Theta(\bar z)}} \overline{h(\bar z)},
\end{equation}
we obtain the analytic continuation of $f$ on $\{z\in\mathbb C: \Im(z)>-\varepsilon\}$. 

\begin{lem}\label{lem:analytic-continuation}
Let $\Theta$ be a meromorphic inner function such that $\Theta'\in L^\infty(\mathbb R)$. Let $\varepsilon>0$ satisfying $\varepsilon<\|\Theta'\|_\infty^{-1}$. Then, for every $f\in \mathcal K_\Theta$, we have
$$
\sup_{|y|<\varepsilon}\int_{-\infty}^{+\infty}|f(x+iy)|^2\,dx\leq \frac{\|f\|_2^2}{(1-\varepsilon \|\Theta'\|_\infty)^2}.
$$
\end{lem}
\begin{proof}
It is clear that 
$$
\sup_{0\leq y<\varepsilon}\int_{-\infty}^{+\infty}|f(x+iy)|^2\,dx\leq \|f\|_2^2.
$$
Now assume that $-\varepsilon<y<0$. According to \eqref{eq2-Theta-minoration1} and \eqref{eq1-analytic-continuation}, we have 
\begin{eqnarray*}
|f(x+iy)|^2&=&\frac{|h(x-iy)|^2}{|\Theta(x-iy)|^2}\\
&\leq & \frac{|h(x-iy)|^2}{(1-\varepsilon \|\Theta'\|_\infty)^2}.
\end{eqnarray*}
Hence,  
\begin{eqnarray*}
\int_{-\infty}^{+\infty}|f(x+iy)|^2\,dx&\leq & \frac{1}{(1-\varepsilon \|\Theta'\|_\infty)^2} \int_{-\infty}^{+\infty}|h(x-iy)|^2\,dx \\
&\leq & \frac{\|h\|_2^2}{(1-\varepsilon \|\Theta'\|_\infty)^2}\\
&=&  \frac{\|f\|_2^2}{(1-\varepsilon \|\Theta'\|_\infty)^2}.
\end{eqnarray*}
\end{proof}

Recall that for a MIF $U$ and $\lambda\in\mathbb R$, the function 
$$
K_\lambda^U(z)=\frac{i}{2\pi}\frac{1-\overline{U(\lambda)}U(z)}{z-\bar\lambda}
$$
is the reproducing kernel of the space $\mathcal K_U$ and $2\pi \|K_\lambda^U\|_2^2=|U'(\lambda)|$. Hence, the function 
$$
k_\lambda^U(z)=\frac{i}{\sqrt{2\pi |U'(\lambda)|}} \frac{1-\overline{U(\lambda)}U(z)}{z-\bar\lambda}
$$
is the normalized reproducing kernel of $\mathcal K_U$. 

\begin{lem}\label{Lem:identite-clee}
Let $I$ and $\Theta$ be two MIF. Assume that $\{I=1\}=\{\lambda_n:n\geq 1\}$ and denote by $\eta_n=|I'(\lambda_n)|^{1/2} |\Theta'(\lambda_n)|^{-1/2}$, $n\geq 1$. Then, for every finitely supported sequence   of complex numbers $(a_n)_{n\geq 1}$, we have 
\begin{equation}
(1-I)\sum_{n\geq 1}a_n k_{\lambda_n}^\Theta=\sum_{n\geq 1}a_n \eta_n k_{\lambda_n}^I -\Theta \sum_{n\geq 1}a_n \eta_n \overline{\Theta(\lambda_n)} k_{\lambda_n}^I.
\end{equation}
\end{lem}

\begin{proof}
We have 
\begin{eqnarray*}
(1-I(z))k_{\lambda_n}^\Theta(z)&=&\frac{i}{\sqrt{2\pi}}\frac{(1-I(z))}{|\Theta'(\lambda_n)|^{1/2}}\frac{1-\overline{\Theta(\lambda_n)}\Theta(z)}{z-\bar\lambda_n}\\
&=&Ê\frac{|I'(\lambda_n)|^{1/2}}{|\Theta'(\lambda_n)|^{1/2}} (1-\overline{\Theta(\lambda_n)}\Theta(z)) \frac{i}{\sqrt{2\pi |I'(\lambda_n)|}}\frac{1-\overline{I(\lambda_n)}I(z)}{z-\bar\lambda_n}\\
&=& \eta_n (1-\overline{\Theta(\lambda_n)}\Theta(z)) k_{\lambda_n}^I(z).
\end{eqnarray*}
Hence 
$$
(1-I)k_{\lambda_n}^\Theta=\eta_n k_{\lambda_n}^I -\Theta \eta_n  \overline{\Theta(\lambda_n)} k_{\lambda_n}^I.
$$
\end{proof}
Assume that the real sequence $(\lambda_n)_{n\geq 1}$ satisfies the following two conditions:
\begin{equation}\label{eq:lambda_tecn1}
\delta:=\inf_{n\geq 1}(\lambda_{n+1}-\lambda_n)>0,
\end{equation}
and 
\begin{equation}\label{eq:lambda_tecn2}
\sup_{n}\left|\sum_{k\neq n}\left(\frac{1}{\lambda_n-\lambda_k}+\frac{\lambda_k}{\lambda_k^2+1}\right)\right|<\infty.
\end{equation}
Define the measure $\mu=\sum_{k=1}^\infty \nu_k \delta_{\lambda_k}$, where $(\nu_k)_{k\geq 1}$ is a sequence of positive real numbers satisfying $\sum_k \frac{\nu_k}{1+\lambda^2_k}<\infty$ and $\sup_k |\nu_k|<\infty$.
Let $G$ be the function defined by 
$$
G(z)=\sum_{k\geq 1}\nu_k \left(\frac{1}{\lambda_k-z}-\frac{\lambda_k}{\lambda_k^2+1}\right).
$$
Clearly, $G$ is analytic in $\mathbb C_+$ and $\Im(G(z))>0$, $z\in\mathbb C_+$. Therefore, the function 
\begin{equation}\label{eq:definition-I}
I=\frac{G-i}{G+i}
\end{equation}
is a meromorphic inner function in the upper half-plane and it is easy to see that $\{I=1\}=\{\lambda_n:n\geq 1\}$. The measure $\mu$ is the so-called Clark measure for $I$ and the sequence $(k_{\lambda_n}^I)_{n\geq 1}$ is an orthonormal basis for $\mathcal K_I$, see \cite{Clark}. Moreover, according to \cite[Lemma 5.2]{baranov-2006}, we have $I'\in L^\infty(\mathbb R)$ and $|I'(\lambda_n)|=2/\nu_n$. Recall that if $I$ is given by $I(z)=e^{iaz} B_Z(z) $, where $a\geq 0$ and $Z=(z_n)_n$ is the zero sequence of $I$, and $B_Z$ is the Blaschke product associated to $Z$, then for $t\in\mathbb R$, we have 
$$
|I'(t)|=a+2\sum_{k\geq 1}\frac{\Im(z_k)}{|t-z_k|^{2}}.
$$
Since $I'\in L^\infty(\mathbb R)$, Lemma~\ref{Lem:mean-value-thm} implies $\delta_1:=\inf_n(\Im(z_n))>0$. Now, it is not difficult to see that if $|t-\lambda_n|\leq \delta$ (where $\delta$ is the constant in \eqref{eq:lambda_tecn1}), then 
$$
\frac{1}{2(1+\frac{\delta^2}{\delta_1^2})}\leq \frac{|I'(t)|}{|I'(\lambda_n)|}=\frac{\nu_n}{2}|I'(t)|.
$$
Since $\mathbb R=\bigcup_n [\lambda_n-\delta,\lambda_n+\delta]$ and $\sup_n \nu_n<\infty$, we finally get 
\begin{equation}\label{eq:bounded-below-above}
|I'(t)|\asymp 1,\qquad t\in\mathbb R.
\end{equation}

We now extend a result from \cite{mitkovski2017basis}. Part 1 is already contained in that form in \cite[Lemma 5.4]{baranov-2006}.
\begin{lem}\label{Lem:bounded-below-T-1-I}
Let $\Lambda=(\lambda_n)_{n\geq 1}$ be a real sequence satisfying \eqref{eq:lambda_tecn1} and \eqref{eq:lambda_tecn2} and let $I$ be the meromorphic inner function defined by \eqref{eq:definition-I}. Let $\Theta$ be a meromorphic inner function such that $\Theta'\in L^\infty(\mathbb R)$.
\begin{enumerate}
\item If $f\in \mathcal K_\Theta$ and $f(\lambda_n)=0$, $n\geq 1$, then $f/(1-I)\in \mathcal K_\Theta$.
\item The Toeplitz operator $T_{1-I}:\mathcal K_\Theta\longrightarrow H^2$ is bounded below. 
\end{enumerate}
\end{lem}
Here, since $1-I$ is analytic, the Toeplitz operator $T_{1-I}$ is just the multiplication by $1-I$. 
\begin{proof}
As already mentioned (1) is exactly Lemma 5.4 from \cite{baranov-2006}. Let us now prove (2), that is 
$$
\inf_{f\in \mathcal K_\Theta,\|f\|_2=1}\|(1-I)f\|_2>0. 
$$
Assume on the contrary that there exists a sequence $(f_k)_{k\geq 1}$ in $\mathcal K_\Theta$, $\|f_k\|_2=1$, such that $\|(1-I)f_k\|_2\to 0$ as $k\to \infty$. Denote by $\phi$ an increasing branch of the argument of $I$. Since $\phi'(t)=|I'(t)|$, according to \eqref{eq:bounded-below-above}, there exists $c_1,c_2>0$ such that 
$$
c_1\leq \phi'(t)\leq c_2,\qquad t\in\mathbb R.
$$
Let us fix $\varepsilon>0$ such that 
$$ 
\varepsilon<\min\left(\frac{2\pi}{c_2},\frac{\delta}{6},\frac{1}{4\|\Theta'\|_\infty}\right).
$$
Then, for every $n\geq 1$, we have 
$$
0<c_1\varepsilon\leq \phi(\lambda_n+\varepsilon)-\phi(\lambda_n)=\int_{\lambda_n}^{\lambda_n+\varepsilon}\phi'(t)\,dt\leq c_2 \varepsilon<2\pi,
$$
and 
$$
0<c_1\varepsilon\leq \phi(\lambda_n)-\phi(\lambda_n-\varepsilon)=\int_{\lambda_n-\varepsilon}^{\lambda_n}\phi'(t)\,dt\leq c_2 \varepsilon<2\pi.
$$
Since $\{\lambda_n\}=\{t\in \mathbb R:e^{i\phi(t)}=1\}$, it is now easy to see that there exists $c_3>0$ such that for all $t\in\mathbb R$ satisfying $\mbox{dist}(t,\Lambda)\geq\varepsilon$, we have 
$$
|I(t)-1|\geq c_3.
$$
Let $E_\varepsilon=\bigcup_{n\geq 1}(\lambda_n-\varepsilon,\lambda_n+\varepsilon)$. Therefore, we get
$$
\int_{\mathbb R\setminus E_\varepsilon}|1-I(t)|^2|f_k(t)|^2\,dt\geq c_3^2 \int_{\mathbb R\setminus E_\varepsilon}|f_k(t)|^2\,dt.
$$
Choose now $k$ sufficiently large so that 
$$
\|(1-I)f_k\|_2\leq \frac{c_3}{\sqrt{2}},
$$
which gives 
$$
\int_{\mathbb R\setminus E_\varepsilon}|f_k(t)|^2\,dt\leq \frac{1}{2}.
$$
Using $1=\|f_k\|_2$, we then get 
$$
\frac{1}{2}\leq \int_{E_\varepsilon}|f_k(t)|^2\,dt.
$$
According to Lemma \ref{lem:analytic-continuation}, we know that $f_k$ is analytic on $\{z\in\mathbb C:\Im(z)>-4\varepsilon\}$ and 
$$
\sup_{|y|<4\varepsilon}\int_{-\infty}^{+\infty}|f_k(x+iy)|^2\,dx\leq \frac{1}{(1-4\varepsilon\|\Theta'\|_{\infty})^2}.
$$
Now let $t\in (-3\varepsilon,3\varepsilon)$. By subharmonicity of $|f_k|^2$, we have
\begin{eqnarray*}
\sum_{n\geq 1}|f_k(\lambda_n+t)|^2 &\leq  &\sum_{n\geq 1} \frac{1}{9\pi\varepsilon^2} \int_{B(\lambda_n+t,3\varepsilon)} |f_k(z)|^2\,dA(z) \\
&\leq & \sum_{n\geq 1} \frac{1}{9\pi\varepsilon^2}  \int_{-3\varepsilon}^{3\varepsilon} \int_{\lambda_n+t-3\varepsilon}^{\lambda_n+t+3\varepsilon}|f_k(x+iy)|^2\,dx\,dy. 
\end{eqnarray*}
\end{proof}
Using the fact that $|\lambda_{n+1}-\lambda_n|\geq \delta>6\varepsilon$, it follows that $[\lambda_{n}+t-3\varepsilon,\lambda_{n}+t+3\varepsilon]$ are disjoints intervals which gives
\begin{eqnarray*}
\sum_{n\geq 1}|f_k(\lambda_n+t)|^2 &\leq  & \frac{1}{9\pi\varepsilon^2}  \int_{-3\varepsilon}^{3\varepsilon}  \int_{-\infty}^{+\infty} |f_k(x+iy)|^2\,dx\,dy \\
&\leq & \frac{2}{3\pi\varepsilon} \frac{1}{(1-4\varepsilon \|\Theta'\|_\infty)^2}.
\end{eqnarray*}
Therefore, 
\begin{eqnarray*}
\frac{1}{2}\leq \int_{E_\varepsilon}|f_k(t)|^2\,dt&=&\sum_{n\geq 1} \int_{-\varepsilon}^{\varepsilon}|f_k(t+\lambda_n)|^2\,dt \\
&\leq & \frac{4}{3\pi}\frac{1}{(1-4\varepsilon \|\Theta'\|_\infty)^2}.
\end{eqnarray*}
Hence $(1-4\varepsilon \|\Theta'\|_\infty)^2\leq 8/(3\pi)$, and letting $\varepsilon\to 0$, we get a contradiction, which proves that the operator $T_{1-I}$ is bounded below as an operator from $\mathcal K_\Theta$ to $H^2$.

\section{Toeplitz operators of the form unitary plus compact}
In this section we revisit some classical results concerning the Toeplitz operators. As we will see our characterization of AOB we get in Section 5 will involve the condition that a certain Toeplitz operator is of the form unitary plus compact. It turns out that this property of Toeplitz operators has been characterized by Douglas \cite{douglas1973banach}. In this section, we revisit his result and explain how to derive Douglas's result from Hartman's theorem on the compactness of Hankel operators \cite{hartman1958on} and the Devinatz-Widom criterion for the invertibility of Toeplitz operators \cite{nikolski1986treatise,nikolski2009functions}. Most of this can be found in \cite{Peller} or \cite{nikolski1986treatise} in the context of the disc but we need the analogues for the upper half-plane. So we give some details for completeness. 
%
%

First, let us recall that if $u\in L^\infty(\mathbb R)$, then the \emph{Toeplitz operator} $T_u$ on the Hardy space of the upper half-plane $\mathcal H^2$ is defined by $T_u(f)=P_+(uf)$, $f\in\mathcal H^2$, where $P_+$ is the Riesz projection from $L^2(\mathbb R)$ onto $\mathcal H^2$. Its cousin, the \emph{Hankel operator}, $H_u$ is defined as an operator from $\mathcal H^2$ into $\mathcal H^2_-=L^2(\mathbb R)\ominus\mathcal H^2$ by $H_u(f)=P_-(uf)$, $f\in\mathcal H^2$, where $P_-=Id-P_+$. 

A well-known theorem of Hartman \cite[Corollary 8.5]{Peller} says that $H_u$ is compact if and only if $u\in H^\infty+C(\dot\R)$, where $H^\infty$ is the space of analytic and bounded functions of $\mathbb C_+$ and 
$$
C(\dot\R)=\{f \in C(\R): \displaystyle\lim_{x \rightarrow \pm \infty} f(x) \mbox { exist  and are equal}\}.
$$
It turns out that the space $H^\infty+C(\dot\R)$ is indeed a closed  subalgebra of $L^\infty(\R)$ (which is not obvious at the first sight but can be deduced from Hartman's theorem). We also recall a part of the Devinatz-Widom criterion for the invertibility of the Toeplitz operator. For the complete result, we refer the reader to (\cite{nikolski2009functions}, page 250). Let us recall that if $b\in L^1_\Pi$, that is 
$$
\int_{-\infty}^{\infty}\frac{|b(t)|}{1+t^2}\,dt<\infty,
$$ 
the Hilbert transform of $b$ is defined as the singular integral
$$
\tilde b(x)=\lim_{\varepsilon\to 0}\frac{1}{\pi}\int_{|x-t|>\varepsilon}\left(\frac{1}{x-t}+\frac{t}{1+t^2}\right)b(t)\,dt.
$$
\begin{thm}[Devinatz-Widom]
Let $u\in L^\infty(\mathbb R)$, $|u|=1$ a.e. on $\mathbb R$. The following are equivalent. 
\begin{enumerate}
\item The Toeplitz operator $T_u$ is invertible; 
\item there exist real valued bounded functions $a,b$ on $\mathbb R$ where $\|a\|_{\infty}<\pi/2$ and a real constant $c$ such that $u = e^{i(c+a+\tilde b)}$.
\item There is an outer function $h \in H^\infty$ such that $\|u-h\|_\infty < 1$
\end{enumerate}
\end{thm}

Note that in \cite{nikolski2009functions}, the result is expressed in the context of the unit disc $\mathbb D$. To transfer the result to the upper-half plane, it is sufficient to use the conformal map (Cayley transform) $\phi(z)=(z-i)/(z+i)$ from $\mathbb C_+$ onto $\mathbb D$, the map $\mathcal M(f)(t)=\frac{\sqrt 2}{t+i}(f\circ \phi)(t)$ for $f\in H^2(\mathbb D)$, which is a unitary map from the Hardy space of the unit onto the Hardy space of the upper-half plane, and note that $T_u=\mathcal M T_{u\circ\phi^{-1}}\mathcal M^{-1}$. 

The following lemma gives a sufficient condition for invertibility. 
\begin{lem}\label{invertible}
If $u=e^{i(\alpha + \tilde\beta)}$, for some real-valued functions $\alpha,\beta\in C(\dot\R)$, then $T_{u}$ is invertible.
\end{lem}
\begin{proof}
Since $\alpha \in C(\dot\R)$, at $\pm \infty$, it admits the same limit, say $\ell$. Then, we can find a real-valued function $v\in C_c^1(\R)$, the space of compactly supported $C^1$ functions on $\R$, such that 
$$
\|\alpha-\ell-v\|_{\infty}<\pi/2.
$$
Note that since $v\in C^1_c(\R)$, $\tilde v$ is a real-valued continuous function which tends to zero at infinity and we have $\tilde{\tilde v}=-v$ (see \cite[Theorem 14.1 \& Corollary 14.9]{mashreghi}). 
Then,
$$
\alpha + \tilde\beta = \alpha-\ell-v+\widetilde{\beta-\tilde{v}}+\ell.
$$
Hence 
$$
u=e^{i\ell} e^{i(\alpha-\ell-v+\widetilde{\beta-\tilde{v}})}.
$$
We can apply the criterion of Devinatz-Widom to conclude that $T_u$ is invertible.  
\end{proof}

When we try to characterize the property for a Toeplitz operator to be unitary modulo the compacts, the class $QC$ of \emph{quasicontinuous functions} naturally appears:
$$
QC=(H^\infty+C(\dot\R))\cap \overline{(H^\infty+C(\dot\R))}.
$$

\begin{lem}[Douglas]\label{unimodcom}
Let $u \in L^\infty$, $|u|=1$ a.e. on $\mathbb R$. Then, $T_u$ is unitary modulo the compacts iff $u \in QC$ and $T_u$ is invertible.
\end{lem}
\begin{proof}
First, let us assume that $u\in QC$ and $T_u$ is invertible. Using the polar decomposition of bounded invertible operators, we can write $$ T_u = U |T_u|,$$ where $U$ is a unitary operator and $|T_u|=(T^*_u T_u)^{\frac{1}{2}}$. Furthermore, we notice the following equality
$$ I - (T^*_u T_u)^{\frac{1}{2}} = (I + (T^*_u T_u)^{\frac{1}{2}})^{-1} (I-T^*_uT_u), $$
using the positivity of $(T^*_u T_u)^{\frac{1}{2}}$ to ensure that the operator $I + (T^*_u T_u)^{\frac{1}{2}}$ is invertible. We recall that $I-T^*_uT_u=H^*_uH_u$ and using Hartman's theorem, we get that $I-(T^*_uT_u)^{\frac{1}{2}}$ is compact. Hence, $|T_u|$ is identity modulo the compacts, ensuring that $T_u=U|T_u|$ is unitary modulo the compacts.

Conversely, let $T_u=U+K$, where $U$ is unitary and $K$ is compact. Then, $T_u^* = U^* + K^*$ and 
\begin{eqnarray*}
T_uT^*_u &=& (U+K)(U^*+K^*)\\
&=& I+K_1,
\end{eqnarray*}
where $K_1=UK^* + KU^* + KK^*$ is compact. Hence, $I-T_uT_u^*$ is compact. Similarly, we can show that $I-T_u^*T_u$ is compact. Thus, $H_u^*H_u= I-T_u^*T_u$ and $H^*_{\overline u}H_{\overline u}= I- T_uT^*_u$ are both compact. But we recall the fact that an operator $T$ is compact iff $T^*T$ is compact. This means that both $u$ and $\overline u$ belong to the space $H^\infty + C(\dot\R)$. That is to say, $u\in QC$. It remains to prove that $T_u$ is invertible. The representation $T_u = U+K$ allows us to reason that since $U$ is Fredholm with index $0$ and $K$ is compact, then $T_u$ is also Fredholm with index $0$. Similarly, $T^*_u$ is also Fredholm with index $0$. We now use Coburn's lemma to claim that either $\ker T_u=\{0\}$ or $\ker T^*_u=\{0\}$. Thus, either and hence both $T_u$ and $T^*_u$ are invertible. 
\end{proof}

We revisit the following theorem of Sarason's \cite{sarason1973algebras} that completely characterizes the unimodular functions that belong to the space QC. We note that the Cayley transform $\phi(z)=(z-i)/(z+i)$ is also a Blaschke factor vanishing at $i$. 
\begin{lem}[Sarason]\label{sarason}
Let $u\in L^\infty(\R)$ be a unimodular function. Then, $u \in QC$ if and only if $u=\phi^ne^{i(a+\tilde b)}$, for some $n\in \mathbb Z$ and two real valued functions $a,b \in C(\dot\R)$.
\end{lem}
\begin{proof}
Let us assume that $u$ is of the form $\phi^ne^{i(a+\tilde b)}$. We can rewrite $u$ as follows. $$u=\phi^ne^{(b+i\tilde b)+(ia-b)}, $$
where $\phi^ne^{ia-b} \in C(\dot\R)$, $e^{b+i\tilde b} \in H^\infty$. Since $ H^\infty + C(\dot\R)$ is an algebra, we have that $u \in H^\infty + C(\dot\R)$. In a similar manner, we can show that $\overline u \in  H^\infty + C(\dot\R)$. Therefore, $u \in QC$.

Conversely, assume that $u \in QC$. We can easily check that the following relationships are true. 
\begin{eqnarray*}
I-T_uT_{\overline u} &=& H^*_{\bar u}H_{\bar u}\\
I-T_{\overline u}T_u &=& H^*_{ u}H_{u}.
\end{eqnarray*}
Since $u,\bar u\in H^\infty + C(\dot\R)$, by Hartman's theorem $H_u$ and $H_{\bar u}$ are compact. This means that $I-T_uT_{\bar u}$ and $I-T_{\bar u}T_u$ are both compact. Thus, $T_u$ is Fredholm. Let $n=-$ind${T_u}$. First let us assume that $n\geq 0$. Let us define a new function $$v:=\overline{\phi^n} u. $$ Then, $T_v=T_{\overline{\phi^n}}T_u$ is also Fredholm. Moreover, $$\mbox{ind}(T_{v})=\mbox{ind}(T_{\overline{\phi^n}}) +  \mbox{ind}(T_u)=n-n=0.$$
By Coburn's result, either $\ker T_v=\{0\}$ or $\ker T_{\bar v}=\{0\}$. Since the index is zero, we finally deduce that both kernels are trivial and  $T_v$ will be invertible. Thus, by the Devinatz-Widom theorem, there is an outer function $h\in H^\infty$ such that $$ \|v-h\|_{\infty}=\|1-\bar v h\|_{\infty}<1. $$ Consequently, $\bar v h$ has a logarithm in the Banach algebra $H^\infty + C(\dot\R)$ (see Lemma 2.13, page 34 in \cite{douglas1972banach}). Therefore, there is a $g\in H^\infty$ and $f\in C(\dot\R)$ such that 
$$\frac{1}{\bar v h}=\frac{v}{h}=e^{f+g}. $$
Hence, $$v=he^{f+g}=e^{ic+\log|h|+i\widetilde{log|h|}+f+g},$$
for some $c\in \R.$ Since $|v|=1$ a.e. on $\mathbb R$, it follows that $\log|h|+\Re f+\Re g=0.$ Let us denote $$b:=\log|h| + \Re(g) = -\Re(f) \in C(\dot\R). $$
Since, $g \in H^\infty$, we have 
$$\tilde b=\widetilde{\log|h|}+\Im(g)-\Im(g(i)). $$
Denoting $a:=\Im(f) + c + \Im(g(i)) \in C(\dot\R)$. Finally, we have that $$u=\phi^n v = \phi^n e^{i(a+\tilde b)}.
$$ In the case $n<0$, we let $v=\phi^n u$ and follow the same argument as before to arrive at the same conclusion.
\end{proof}

Finally, we can prove the following result.
\begin{thm}\label{umodcom2}
Let $u\in L^\infty(\R)$, $|u|=1$ a.e. on $\mathbb R$. The following are equivalent. 
\begin{enumerate}
\item $T_u$ is unitary modulo the compacts.
\item There exist real valued functions $a,b \in C(\dot\R)$ such that $u=e^{i(a+\tilde b)}$.
\end{enumerate}
\end{thm}
\begin{proof}
First let us assume that there exist real valued functions $a,b \in C(\dot\R)$ such that $u=e^{i(a+\tilde b)}$. Using Lemma \ref{invertible}, we can conclude that $T_u$ is invertible. Furthermore, using Lemmas \ref{sarason} and \ref{unimodcom}, we can conclude that $T_u$ is indeed unitary modulo the compacts. 

Conversely, assume that $T_u$ is unitary modulo the compacts. It follows from Lemma \ref{unimodcom} that $u\in QC$ and $T_u$ is invertible. Next, using Lemma \ref{sarason}, we see that $u=\phi^ne^{i(a+\tilde b)},$ for some real valued functions $a,b \in C(\dot\R)$ and some $n\in \mathbb Z$. It remains to prove that $n=0$. Let us denote by $u_1 = e^{i(a+ \tilde b)}$. By Lemma \ref{invertible}, the Toeplitz operator $T_{u_1}$ is invertible. If $n>0$, then 
$$T_u = T_{u_1 \phi^n} = T_{u_1}T_{\phi^n}, $$
whence $T_{\overline{\phi^n}}=T^*_{\phi^n} = T^*_u (T^{-1}_{u_1})^*$ should also be invertible. But this is absurd because $\ker T_{\overline{\phi^n}}= K_{\phi^n}$. If $n<0$, then $T_u= T_{\phi^n}T_{u_1}$, that is to say 
$$ T_{\phi^n}=T_u (T_{u_1})^{-1}$$ is invertible, which is also absurd because $\ker T_{\phi^n}=K_{\phi^{-n}}.$
\end{proof}

\section{Riesz bases}
In this section we revisit a recent work of Mitkovski \cite{mitkovski2017basis} on systems of exponential systems which form a Riesz basis. We provide a generalization of his result, while maintaining the techniques of the proofs. 

\begin{thm}\label{Thm:analogue-mishko}
Let $(\lambda_n)_{n\geq 1}$ be a real sequence satisfying \eqref{eq:lambda_tecn1} and \eqref{eq:lambda_tecn2} and let $I$ be the meromorphic inner function defined by \eqref{eq:definition-I}. Let $\Theta$ be a meromorphic inner function such that $\Theta'\in L^\infty(\mathbb R)$ and $\inf_{n\geq 1}|\Theta'(\lambda_n)|>0$. Then,
\begin{enumerate}
\item $(k_{\lambda_n}^\Theta)_{n\geq 1}$ is complete in $\mathcal K_\Theta$ if and only if $T_{\Theta \bar I}:\HH^2\longrightarrow \HH^2$ has a dense range.
\item $(k_{\lambda_n}^\Theta)_{n\geq 1}$ is a Riesz sequence in $\mathcal K_\Theta$ if and only if $T_{\Theta \bar I}:\HH^2\longrightarrow \HH^2$ is injective with closed range.
\item $(k_{\lambda_n}^\Theta)_{n\geq 1}$ is a Riesz basis for $\mathcal K_\Theta$ if and only if $T_{\Theta \bar I}:\HH^2\longrightarrow \HH^2$ is invertible.
\end{enumerate}
\end{thm}

\begin{proof}
Let us first show that the upper inequality in Riesz basis property \eqref{eq:defn-RS} is satisfied for $(k_{\lambda_n}^\Theta)_n$ under the hypothesis of the theorem. Indeed, according to Lemma~\ref{Lem:identite-clee}, we have 
$$
(1-I)\sum_{n\geq 1}a_n k_{\lambda_n}^\Theta=\sum_{n\geq 1}a_n\eta_n k_{\lambda_n}^I -\Theta \sum_{n\geq 1}a_n \eta_n \overline{\Theta(\lambda_n)} k_{\lambda_n}^I,
$$
where $\eta_n=|I'(\lambda_n)|^{1/2} |\Theta'(\lambda_n)|^{-1/2}\asymp 1$ by hypothesis and \eqref{eq:bounded-below-above}. 
Using Lemma~\ref{Lem:bounded-below-T-1-I}, 
\begin{eqnarray*}
\left\|\sum_{n\geq 1}a_n k_{\lambda_n}^\Theta\right\|_2&\lesssim& \left\|(1-I)\sum_{n\geq 1}a_n k_{\lambda_n}^\Theta\right\|_2\\
&\leq & \left\|\sum_{n\geq 1}a_n \eta_n k_{\lambda_n}^I\right\|_2+\left\|\sum_{n\geq 1}a_n \eta_n \overline{\Theta(\lambda_n)}k_{\lambda_n}^I\right\|_2.
\end{eqnarray*}
Then, using the fact that $(k_{\lambda_n}^I)_{n\geq 1}$ is an orthonormal basis for $\mathcal K_I$ and $|\Theta(\lambda_n)|=1$, $n\geq 1$, we get 
\begin{equation}\label{eq:upper-inequality-rb}
\left\|\sum_{n\geq 1}a_n k_{\lambda_n}^\Theta\right\|_2  \lesssim \left(\sum_{n\geq 1}|a_n|^2\right)^{1/2}.
\end{equation}
\\

Let us  now prove (1). First, assume that $T_{\Theta\bar I}$ has dense range and assume that there exists $f\in \mathcal K_\Theta$, $f(\lambda_n)=0$, $n\geq 1$, and $f\not\equiv 0$. According to Lemma~\ref{Lem:bounded-below-T-1-I}, the function $g:=f/(1-I)$ belongs to $\mathcal K_\Theta$ and, since $\mathcal K_\Theta=\HH^2\cap\Theta\overline{\HH^2}$, $\bar\Theta Ig=-\bar\Theta (1-I)g+\bar\Theta g=-\bar\Theta f+\bar\Theta g\in\overline{\HH^2}$. In particular, 
$T_{I\bar\Theta}g=0$ and $T_{I\bar\Theta}$ is not injective which contradicts the fact that $T_{\Theta\bar I}=T_{I\bar\Theta}^*$ has dense range. Conversely, suppose that $T_{\Theta\bar I}$ does not have a dense range. Then, there is a function $g\in\ker(T_{I\bar\Theta})$, $g\not\equiv 0$. Hence $Ig\bar\Theta\in\overline{\HH^2}$, which gives that both $g$ and $Ig$ belong to $\mathcal K_\Theta$. So $f:=(1-I)g$ is also in $\mathcal K_\Theta$, $f\not\equiv 0$ and 
$$
f(\lambda_n)=(1-I(\lambda_n))g(\lambda_n)=0,\qquad n\geq 1,
$$
which proves that $(k_{\lambda_n}^\Theta)_{n\geq 1}$ is not complete in $\mathcal K_\Theta$. 
\\

Let us now prove (2). Assume that $(k_{\lambda_n}^\Theta)_{n\geq 1}$ is not a Riesz sequence in $\mathcal K_\Theta$. According to \eqref{eq:upper-inequality-rb}, it means that the lower inequality in \eqref{eq:defn-RS} is not satisfied. Then, for any $\varepsilon>0$, there exists $\{b_n\}_{n\geq 1}\in\ell^2$ such that 
$$
\left\|\sum_{n\geq 1}b_n k_{\lambda_n}^\Theta\right\|^2<\varepsilon \sum_{n\geq 1}|b_n|^2.
$$
Take $a_n=b_n/\|b_n\|_{\ell^2}$. Then $\sum_{n\geq 1}|a_n|^2=1$ and 
$$
\left\|\sum_{n\geq 1}a_n k_{\lambda_n}^\Theta\right\|^2<\varepsilon. 
$$
In particular, we can construct a sequence $(a_n^\ell)_{n\geq 1}$ of elements of $\ell^2$ such that for every $\ell\geq 1$, $\sum_{n\geq 1}|a_n^\ell|^2=1$ and 
$$
\left\|\sum_{n\geq 1}a_n^\ell k_{\lambda_n}^\Theta\right\|_2\to 0,\quad \mbox{as }\ell\to\infty.
$$
Let $g_\ell:=\sum_{n\geq 1}a_n^\ell k_{\lambda_n}^\Theta$. Then $g_\ell\in \mathcal K_\Theta$ and $\|g_\ell\|_2\to 0$ as $\ell\to\infty$. According to Lemma~\ref{Lem:identite-clee}, we have 
$$
(1-I)g_\ell=\sum_{n\geq 1}a_n^\ell \eta_n k_{\lambda_n}^I -\Theta\sum_{n\geq 1}a_n^\ell \eta_n \overline{\Theta(\lambda_n)}k_{\lambda_n}^I.
$$
Multiply by $\bar I$ and take the Riesz projection gives 
$$
P_+(\bar I(1-I)g_\ell)=\sum_{n\geq 1}a_n^\ell \eta_n P_+(\bar I k_{\lambda_n}^I)-P_+(\Theta\bar I\sum_{n\geq 1}a_n^\ell \eta_n \overline{\Theta(\lambda_n)}k_{\lambda_n}^I).
$$
But $k_{\lambda_n}^I\in \mathcal K_I=\HH^2\cap I\overline{\HH^2}$, hence $P_+(\bar I k_{\lambda_n}^I)=0$ and, then 
$$
P_+(\bar I(1-I)g_\ell)=-T_{\Theta \bar I}(\sum_{n\geq 1}a_n^\ell \eta_n \overline{\Theta(\lambda_n)}k_{\lambda_n}^I).
$$
On one hand, we have
$$
\|P_+(\bar I(1-I)g_\ell)\|_2\to 0,\quad \mbox{as }\ell\to\infty.
$$
And on the other hand, since $\Theta'\in L^\infty(\mathbb R)$, we also have
\begin{eqnarray*}
\left\|\sum_{n\geq 1}a_n^\ell \eta_n \overline{\Theta(\lambda_n)}k_{\lambda_n}^I \right\|_2^2&=&\sum_{n\geq 1}|a_n^\ell|^2 \eta_n^2 |\Theta(\lambda_n)|^2 \\
&=&\sum_{n\geq 1}|a_n^\ell|^2 \eta_n^2 \\
&\gtrsim & \sum_{n\geq 1}|a_n^\ell|^2=1,
\end{eqnarray*}
which proves that $T_{\Theta\bar I}$ is not bounded below.  Conversely, assume that $T_{\Theta\bar I}$ is not bounded below. Hence there exists a sequence $(f_\ell)_{\ell\geq 1}$, $f_\ell\in\HH^2$, $\|f_\ell\|_2=1$, such that $\|T_{\Theta\bar I}f_\ell\|_2\to 0$, as $\ell\to \infty$. Let $h_\ell=T_{\Theta \bar I}f_\ell$. Then, $\|h_\ell\|_2\to 0$ as $\ell\to\infty$, and $g_\ell:=\Theta \bar I f_\ell-h_\ell\in \overline{\HH^2}$ with $\|g_\ell\|_2\asymp 1$ (because $\|f_\ell\|_2=1$ and $\|h_\ell\|_2\to 0$). Note that 
$$
Ig_\ell=\Theta f_\ell-Ih_\ell\in \HH^2\cap I\overline{\HH^2}=\mathcal K_I.
$$
Since $(k_{\lambda_n}^I)_{n\geq 1}$ is an orthonormal basis for $\mathcal K_I$, there exists $(g_n^\ell)_{n\geq 1}\in\ell^2$ such that 
$$
Ig_\ell=\sum_{n\geq 1}g_n^\ell k_{\lambda_n}^I.
$$ 
Note also that
$$
\sum_{n\geq 1}|g_n^\ell|^2=\|Ig_\ell\|_2^2=\|g_\ell\|_2^2\asymp 1.
$$
Then
$$
Ih_\ell=\Theta f_\ell-I g_\ell=\Theta f_\ell-\sum_{n\geq 1}g_n^\ell k_{\lambda_n}^I,
$$
and using one more time Lemma~\ref{Lem:identite-clee}, we get
\begin{equation}\label{eq22334:preuve-thm-riesz-basis}
Ih_\ell=\Theta f_\ell-\left((1-I)\sum_{n\geq 1}\frac{g_n^\ell}{\eta_n}k_{\lambda_n}^\Theta+\Theta\sum_{n\geq 1}g_n^\ell \overline{\Theta(\lambda_n)} k_{\lambda_n}^I\right).
\end{equation}
Thus
$$
\Theta f_\ell-\Theta\sum_{n\geq 1}g_n^\ell \overline{\Theta(\lambda_n)} k_{\lambda_n}^I-\sum_{n\geq 1}\frac{g_n^\ell}{\eta_n}k_{\lambda_n}^\Theta=Ih_\ell-I\sum_{n\geq 1}\frac{g_n^\ell}{\eta_n}k_{\lambda_n}^\Theta.
$$
The difference of the first two terms on the left side is in $\Theta\HH^2$, while the third term is in $\mathcal K_\Theta$. Using orthogonality and triangle's inequality, we obtain
\begin{eqnarray*}
\left\| \Theta f_\ell-\Theta\sum_{n\geq 1}g_n^\ell \overline{\Theta(\lambda_n)} k_{\lambda_n}^I\right\|_2^2+\left\| \sum_{n\geq 1}\frac{g_n^\ell}{\eta_n}k_{\lambda_n}^\Theta\right\|_2^2 &\leq & \left(\|Ih_\ell\|_2+ 
\left\| I \sum_{n\geq 1}\frac{g_n^\ell}{\eta_n}k_{\lambda_n}^\Theta\right\|_2\right)^2 \\
&=&\left(\|h_\ell\|_2+ \left\| \sum_{n\geq 1}\frac{g_n^\ell}{\eta_n}k_{\lambda_n}^\Theta\right\|_2\right)^2.
\end{eqnarray*}
Since by \eqref {eq:upper-inequality-rb}, we have
$$
\left\| \sum_{n\geq 1}\frac{g_n^\ell}{\eta_n}k_{\lambda_n}^\Theta\right\|_2\lesssim \left(\sum_{n\geq 1}\frac{|g_n^\ell|^2}{\eta_n^2}\right)^{1/2}\lesssim 1, 
$$
and since $\|h_\ell\|_2\to 0$ as $\ell\to\infty$, we deduce 
$$
\left\|\Theta f_\ell-\Theta \sum_{n\geq 1}g_n^\ell \overline{\Theta(\lambda_n)}k_{\lambda_n}^I \right\|_2\to 0,\qquad \mbox{as }\ell \to\infty.
$$
Using this and \eqref{eq22334:preuve-thm-riesz-basis}, we obtain
$$
\left\|(1-I)\sum_{n\geq 1}\frac{g_n^\ell}{\eta_n}k_{\lambda_n}^\Theta\right\|_2\to 0,\quad \mbox{as }\ell\to\infty.
$$
Lemma~\ref{Lem:bounded-below-T-1-I} implies now that 
$$
\left\|\sum_{n\geq 1}\frac{g_n^\ell}{\eta_n}k_{\lambda_n}^\Theta\right\|_2\to 0,\quad \mbox{as }\ell\to\infty,
$$
which together with
$$
\sum_{n\geq 1}\frac{|g_n^\ell|^2}{\eta_n^2}\asymp \sum_{n\geq 1}|g_n^\ell|^2\asymp 1,
$$
implies that $(k_{\lambda_n}^\Theta)_{n\geq 1}$ cannot be a Riesz sequence in $\mathcal K_\Theta$. 

The part (3) follows immediately from (1) and (2). 

\end{proof}

\section{Asymptotically orthonormal bases.}

In that section, we give a result on sequences of reproducing kernels which form an AOB. We need to recall the notion of angle between two closed subspaces $\mathcal F,\mathcal G$ of an Hilbert space $\mathcal H$. The angle $\langle \mathcal F,\mathcal G \rangle$ between $\mathcal F$ and $\mathcal G$ is defined by the condition:
$$
\langle \mathcal F,\mathcal G\rangle\in [0,\pi/2]\quad \cos\langle \mathcal F,\mathcal G\rangle=\sup_{f\in\mathcal F,g\in\mathcal G}\frac{|\langle f,g \rangle|}{\|f\|\|g\|}.
$$
It is easy to see that $\cos\langle \mathcal F,\mathcal G\rangle=\|P_{\mathcal F} P_{\mathcal G}\|$, where $P_{\mathcal F}$ and $P_{\mathcal G}$ are the corresponding orthogonal projections. 
\begin{thm}
Let $(\lambda_n)_{n\geq 1}$ be a real sequence satisfying \eqref{eq:lambda_tecn1} and \eqref{eq:lambda_tecn2} and let $I$ be the meromorphic inner function defined by \eqref{eq:definition-I}. Let $\Theta$ be a meromorphic inner function such that $\Theta'\in L^\infty(\mathbb R)$ and $\eta_n=|I'(\lambda_n)|^{1/2} |\Theta'(\lambda_n)|^{-1/2}\to 1$ as $n\to\infty$. Then, the following assertions are equivalent: 
\begin{enumerate}
\item there exists two real-valued functions $\alpha,\beta\in C(\dot\R)$ satisfying 
$$
\arg\Theta-\arg I=\alpha+\widetilde\beta;
$$
\item the operator $T_{\Theta \bar I}$ is of the form unitary plus compact.
\item \begin{enumerate}
\item[(i)] the sequence $(k_{\lambda_n}^\Theta)_{n\geq 1}$ is a complete AOB for $\mathcal K_\Theta$;
\item [(ii)]$
\cos\langle \overline{\mbox{span}}(k_{\lambda_n}^\Theta:n\geq N), I\HH^2\rangle\to 0\quad\mbox{as }N\to\infty.
$
\end{enumerate}
\end{enumerate}
\end{thm}

\begin{proof}
The equivalence between (1) and (2) is proved in Theorem~\ref{umodcom2}. \\

Let us now prove $(2)\Longrightarrow (3)$. So, we assume that $T_{\Theta \bar I}$ is of the form unitary plus compact. We have already seen that necessarily $T_{\Theta \bar I}$ is invertible (see Lemma ~\ref{unimodcom}) and by Theorem~\ref{Thm:analogue-mishko}, we get that the sequence $(k_{\lambda_n}^\Theta)_{n\geq 1}$ is a Riesz basis for $\mathcal K_\Theta$. Observe now that 
$$
(T_{\bar I}|\mathcal K_\Theta)^*=P_\Theta I|\mathcal K_\Theta=\Theta P_- \bar\Theta I |\mathcal K_\Theta=\Theta H_{I\bar\Theta}|\mathcal K_\Theta.
$$
But Lemma~\ref{unimodcom} implies $I\bar\Theta\in QC$ and by Hartman's theorem, we deduce that $H_{I\bar\Theta}$ is compact. Hence $T_{\bar I}|\mathcal K_\Theta$ is compact. Using Lemma~\ref{Lem:identite-clee}, we have 
$$
(1-I)k_{\lambda_n}^\Theta=\eta_n k_{\lambda_n}^I -\Theta \overline{\Theta(\lambda_n)} \eta_n k_{\lambda_n}^I,
$$
where $\eta_n\to 1$ as $n\to\infty$. If we apply the operator $T_{\bar I}$, using the fact that $\ker T_{\bar I}=\mathcal K_I$, we obtain
$$
(Id-T_{\bar I})k_{\lambda_n}^\Theta=T_{\Theta\bar I}(\eta_n \overline{\Theta(\lambda_n)}k_{\lambda_n}^I).
$$
Since $(\eta_n \overline{\Theta(\lambda_n)}k_{\lambda_n}^I)_{n\geq 1}$ is an orthogonal basis for $\mathcal K_I$, since $(k_{\lambda_n}^\Theta)_{n\geq 1}$ is a Riesz basis of $\mathcal K_\Theta$ and $T_{\Theta\bar I}$ is invertible, the last equation implies that the operator $(Id-T_{\bar I})|\mathcal K_\Theta$ is also invertible. Therefore, we can write 
$$
k_{\lambda_n}^\Theta=((Id-T_{\bar I})|\mathcal K_\Theta)^{-1}T_{\Theta\bar I}(\eta_n \overline{\Theta(\lambda_n)}k_{\lambda_n}^I).
$$
Now observe that the operator $((Id-T_{\bar I})|\mathcal K_\Theta)^{-1}T_{\Theta\bar I}$ is of the form unitary plus compact and the sequence $((\eta_n \overline{\Theta(\lambda_n)}k_{\lambda_n}^I)_{n\geq 1}$ is a complete AOB for $\mathcal K_I$. It implies that the sequence $(k_{\lambda_n}^\Theta)_{n\geq 1}$ is a complete AOB for $\mathcal K_\Theta$. To prove (ii), note that 
\begin{eqnarray*}
\cos\langle \overline{\mbox{span}}(k_{\lambda_n}^\Theta:n\geq N), I\HH^2\rangle&=&\left\|P_{I\HH^2}|\overline{\mbox{span}}(k_{\lambda_n}^\Theta:n\geq N)| \right\|\\
&=&\left\|IP_+\bar I |\overline{\mbox{span}}(k_{\lambda_n}^\Theta:n\geq N)\right\|\\
&=&\left\|T_{\bar I} |\overline{\mbox{span}}(k_{\lambda_n}^\Theta:n\geq N)\right\|.
\end{eqnarray*}
Since $(k_{\lambda_n}^\Theta)_{n\geq 1}$ is a complete AOB for $\mathcal K_\Theta$ and $T_{\bar I}|\mathcal K_\Theta$ is compact, we can use \cite[Lemma 10.31]{fricain2016theory} to conclude that 
$$
\left\|T_{\bar I} |\overline{\mbox{span}}(k_{\lambda_n}^\Theta:n\geq N)\right\|\to 0,\quad \mbox{as }N\to\infty.
$$

It remains to prove $(3)\Longrightarrow (2)$. Let us denote by $P_N^\Theta$ the projection from $\mathcal K_\Theta$ onto $\overline{\mbox{span}}(k_{\lambda_n}^\Theta:n\geq N)$. We have
\begin{eqnarray*}
\left\|P_{I\HH^2}P_N^\Theta\left(\sum_{n\geq 1}a_n k_{\lambda_n}^\Theta\right)\right\|_2^2&=&\left\|P_{I\HH^2}\left(\sum_{n\geq N}a_n k_{\lambda_n}^\Theta\right)\right\|_2^2\\
&\leq & \varepsilon_N^2 \left\|\sum_{n\geq N}a_n k_{\lambda_n}^\Theta\right\|_2^2,
\end{eqnarray*}
where $\varepsilon_N=\cos\langle \overline{\mbox{span}}(k_{\lambda_n}^\Theta:n\geq N), I\HH^2\rangle\to 0$ as $N\to\infty$. Moreover, since $(k_{\lambda_n}^\Theta)_{n\geq 1}$ is a complete AOB for $\mathcal K_\Theta$, then
\begin{eqnarray*}
\left\|P_{I\HH^2}P_N^\Theta\left(\sum_{n\geq 1}a_n k_{\lambda_n}^\Theta\right)\right\|_2^2&\lesssim& \varepsilon_N^2\sum_{n\geq N}|a_n|^2\\
&\leq & \varepsilon_N^2\sum_{n\geq 1}|a_n|^2\\
&\lesssim&\varepsilon_N^2 \left\|\sum_{n\geq 1}a_n k_{\lambda_n}^\Theta \right\|_2^2.
\end{eqnarray*}
Therefore, we get $\|P_{I\HH^2}P_N^\Theta\|\lesssim\varepsilon_N\to 0$, $N\to\infty$. Since 
$$
P_{I\HH^2}|\mathcal K_\Theta=P_{I\HH^2}P_N^\Theta+P_{I\HH^2}(Id-P_N^\Theta),
$$
and $Id-P_N^\Theta$ is a finite rank operator, we get that $P_{I\HH^2}|\mathcal K_\Theta$ can be approximated in norm by a sequence of finite rank operators and hence $P_{I\HH^2}|\mathcal K_\Theta$ is compact. We deduce that $T_{\bar I}|K_\Theta=\overline{I}P_{I\mathcal H^2}|\mathcal K_\Theta$ is also compact. Since $(k_{\lambda_n}^\Theta)_n$ is an AOB for $K_\Theta$ and $(\eta_n \overline{\Theta(\lambda_n)}k_{\lambda_n}^I)_n$ is an AOB for $\mathcal K_I$, there is an isomorphism $U:\mathcal K_I\longrightarrow K_\Theta$ of the form unitary plus compact such that $U(\eta_n\overline{\Theta(\lambda_n)}k_{\lambda_n}^I)=k_{\lambda_n}^\Theta$, $n\geq 1$. Using one more time the relation $$
(Id-T_{\bar I})k_{\lambda_n}^\Theta=T_{\Theta\bar I}(\eta_n \overline{\Theta(\lambda_n)}k_{\lambda_n}^I),
$$ 
we then deduce that the operator $T_{\Theta \bar I}:\mathcal K_I\longrightarrow K_\Theta$ coincide with $(Id-T_{\bar I})U$. In particular, it is of the form unitary plus compact. Note now that the operator $T_{\Theta\bar I}:I\mathcal H^2\longrightarrow \Theta\mathcal H^2$ is unitary. Hence, we get that the operator $T_{\Theta\bar I}:\mathcal H^2\longrightarrow \mathcal H^2$ is unitary plus compact.

\end{proof}

\end{document}